\documentclass[10pt,a4paper]{article}
\usepackage{latexsym,amsmath}
\usepackage{times}
\usepackage{amsfonts}
\usepackage{amssymb}
\usepackage{latexsym}
\usepackage{color}
\usepackage{a4wide}

\newcommand\gnm{G(n,m)}
\newcommand\Gnm{G'(n,m)}
\newcommand\MPCPS{Mathematical Proceedings of the Cambridge Philosophical Society}

\newcommand\COMB{Combinatorica}
\def\vec#1{\mathchoice{\mbox{\boldmath$\displaystyle#1$}}
{\mbox{\boldmath$\textstyle#1$}}
{\mbox{\boldmath$\scriptstyle#1$}}
{\mbox{\boldmath$\scriptscriptstyle#1$}}}

\newcommand{\dAN}{d_{k,\mathrm{AN}}}
\newcommand{\dk}{d_{k-\mathrm{col}}}
\newcommand{\dfirst}{d_{k,\mathrm{first}}}
\newcommand{\dsecond}{d_{k,\mathrm{second}}}
\newcommand{\dfreeze}{d_{k,\mathrm{freeze}}}

\DeclareMathOperator{\pr}{P}

\newcommand{\qed}{\hfill$\Box$\smallskip}
\newenvironment{proof}{\emph{Proof.}}{}

\newtheorem{definition}{Definition}[section]

\newtheorem{theorem}[definition]{Theorem}
\newtheorem{lemma}[definition]{Lemma}
\newtheorem{proposition}[definition]{Proposition}
\newtheorem{corollary}[definition]{Corollary}

\newcommand\dist{\mbox{dist}}
\newcommand\cA{\mathcal{A}}
\newcommand\cB{\mathcal{B}}
\newcommand\cC{\mathcal{C}}

\newcommand\cE{\mathcal{E}}
\newcommand\cU{\mathcal{U}}
\newcommand\cN{\mathcal{N}}

\newcommand\cS{\mathcal{S}}
\newcommand\cT{\mathcal{T}}
\newcommand\cI{\mathcal{I}}

\newcommand\cM{\mathcal{M}}

\newcommand\cY{\mathcal{Y}}
\newcommand\cV{\mathcal{V}}
\newcommand\cW{\mathcal{W}}

\def\cC{{\mathcal C}}
\def\cE{{\cal E}}

\newcommand\eul{\mathrm{e}}
\newcommand\eps{\varepsilon}

\newcommand\Erw{\mathrm{E}}

\newcommand{\vecone}{\vec{1}}

\newcommand{\Po}{{\rm Po}}
\newcommand{\Bin}{{\rm Bin}}

\newcommand{\bink}[2] {{{#1}\choose {#2}}}
\newcommand\ra{\rightarrow}

\newcommand\bc[1]{\left({#1}\right)}
\newcommand\cbc[1]{\left\{{#1}\right\}}
\newcommand\bcfr[2]{\bc{\frac{#1}{#2}}}

\newcommand\brk[1]{\left\lbrack{#1}\right\rbrack}

\newcommand\norm[1]{\left\|{#1}\right\|}
\newcommand\abs[1]{\left|{#1}\right|}

\newcommand{\Whp}{W.h.p.}
\newcommand{\whp}{w.h.p.}

\newcommand{\Bollobas}{Bollob\'as}

\newcommand{\Luczak}{\L uczak}

\newcommand\Lem{Lemma}
\newcommand\Prop{Proposition}
\newcommand\Thm{Theorem}
\newcommand\Cor{Corollary}
\newcommand\Sec{Section}

\begin{document}

\title{\bf Upper-bounding the $k$-colorability threshold by counting covers}

\author{
Amin Coja-Oghlan\thanks{{\tt acoghlan@math.uni-frankfurt.de}.
	Goethe University, Mathematics Institute, 10 Robert Mayer St, Frankfurt 60325, Germany.
		The research leading to these results has received funding from the European Research Council under the European Union's Seventh Framework
			Programme (FP/2007-2013) / ERC Grant Agreement n.\ 278857--PTCC.}
}
\date{\today}

\maketitle

\begin{abstract}
\noindent
Let $\gnm$ be the random graph on $n$ vertices with $m$ edges.
Let $d=2m/n$ be its average degree.
We prove that $\gnm$ fails to be $k$-colorable \whp\ if $d>2k\ln k-\ln k-1+o_k(1)$.
This  matches a conjecture put forward on the basis of sophisticated but non-rigorous
statistical physics ideas (Krzakala, Pagnani, Weigt: Phys.\ Rev.\ E {\bf70} (2004)).
The proof is based on applying the first moment method to the number of
``covers'', a physics-inspired concept.
By comparison, a standard first moment over the number of $k$-colorings shows that $\gnm$ is not $k$-colorable
\whp\ if $d>2k\ln k-\ln k$.\\

\noindent
\bigskip
\emph{Key words:}	random structures, phase transitions, graph coloring.
\end{abstract}

\section{Introduction}

{\em Let $\gnm$ be the 
 random graph on $V=\cbc{1,\ldots,n}$ with $m$ edges. 
Unless specified otherwise, we let $m=\lceil dn/2\rceil$ for a number
$d>0$ that remains fixed as $n\ra\infty$.
Let $k\geq3$ be an $n$-independent integer.
We say that $\gnm$ has a property $\cE$ {\bf\em with high probability} (`\whp') if $\lim_{n\ra\infty}\pr\brk{\gnm\in\cE}= 1$.}

\bigskip\noindent
One of the longest-standing open problems in the theory of random graphs is whether there is a phase transition for $k$-colorability in $\gnm$ 
and, if so, at what average degree $d$ it occurs~\cite{AchFried,Cheeseman,ER}.
Regarding existence, Achlioptas and Friedgut~\cite{AchFried} proved that for any $k\geq3$ there is a {\em sharp threshold sequence} $\dk(n)$
such that for any fixed $\eps>0$ the random graph $\gnm$ is $k$-colorable \whp\ if $m/n<(1-\eps)\dk(n)$
and non-$k$-colorable \whp\ if $m/n>(1+\eps)\dk(n)$.
To establish the existence of an actual sharp threshold, one would have to show that the sequence $\dk(n)$ converges.
This is widely conjectured to be the case (explicitly so in~\cite{AchFried}) but as yet unproven.

In any case, the techniques used to prove the existence of $\dk(n)$ shed no light on its location.
An upper bound is easily obtained via the {\em first moment method}.
Indeed, a simple calculation shows that for $k\geq3$ and 
	\begin{equation}\label{eqNaiveFirstMoment}
	d>\dfirst=2k\ln k-\ln k,
	\end{equation}
the expected number number of $k$-colorings tends to $0$ as $n\ra\infty$ (e.g., \cite{AchlioptasMolloy}).
Hence, Markov's inequality implies that $\gnm$ fails to be $k$-colorable for $d>\dfirst$ \whp\
Furthermore, Achlioptas and Naor~\cite{AchNaor} used the second moment method to prove that for any $k\geq3$,
$\gnm$ is $k$-colorable \whp\ if
	\begin{equation}\label{eqAN}
	\dk\geq\dAN=2(k-1)\ln(k-1)=2k\ln k-2\ln k-2+o_k(1).
	\end{equation}
Here and throughout the paper, we use the symbol $o_k(1)$ to hide terms that tend to zero for large $k$.
The bound~(\ref{eqAN}) was recently improved~\cite{ACOVilenchik}, also via a second moment argument, for sufficiently large $k$ to
	\begin{equation}\label{eqSecond}
	\dk\geq\dsecond=2k\ln k-\ln k-2\ln2+o_k(1).
	\end{equation}
This leaves an additive gap of $2\ln 2+o_k(1)$ between the upper bound~(\ref{eqNaiveFirstMoment}) and the lower bound~(\ref{eqSecond}).

The problem of $k$-coloring $\gnm$ is closely related to the ``diluted mean-field $k$-spin Potts antiferromagnet'' model of statistical physics.
Indeed, over the past decade physicists have developed sophisticated, albeit mathematically non-rigorous formalisms 
for identifying phase transitions in random discrete structures, the ``replica method'' and the ``cavity method'' (see~\cite{MM} for details and references).
Applied to the problem of $k$-coloring $\gnm$~\cite{KPW,MPWZ,vanMourik,LenkaFlorent}, these techniques lead to the conjecture that
	\begin{equation}\label{eqcavity}
	\dk=2k\ln k-\ln k-1+o_k(1).
	\end{equation}
The main result of the present paper is an improved {\em upper} on $\dk$ that matches the physics prediction~(\ref{eqcavity})
(at least up to the term hidden in the $o_k(1)$).

\begin{theorem}\label{Thm_main}
We have $\dk\leq2k\ln k-\ln k-1+o_k(1)$.
\end{theorem}

\Thm~\ref{Thm_main} improves the naive first moment bound~(\ref{eqNaiveFirstMoment}) by about an additive $1$.
This proves, perhaps surprisingly, that the $k$-colorability threshold  (if it exists) does {\em not} coincide with the first moment bound.
Furthermore, \Thm~\ref{Thm_main} narrows the gap to the lower bound~(\ref{eqSecond}) to $2\ln2-1+o_k(1)\approx0.39$.

The proof of \Thm~\ref{Thm_main} is based on a concept borrowd from the ``cavity method'',
namely the notion of {\em covers}.
This concept is closely related to hypotheses on the ``geometry'' of the set of $k$-colorings of the random graph,
which are at the core of the cavity method~\cite{LenkaFlorent,MM,pnas,KPW,MPWZ,LenkaThesis}.
More precisely, let $\cS_k(\gnm)\subset\cbc{1,\ldots,k}^n$ be the set of all $k$-colorings of $\gnm$.
According to the cavity method, for average degrees $(1+o_k(1))k\ln k<d<\dk$
\whp\ the set $\cS_k(\gnm)$ has a decomposition
	$$\cS_k(\gnm)=\bigcup_{i=1}^N\cC_i$$
into $N=\exp(\Omega(n))$ non-empty ``clusters''  $\cC_i$ such that for any two colorings $\sigma,\tau$ that belong to distinct clusters we have
	$$\dist(\sigma,\tau)=\abs{\cbc{v\in V:\sigma(v)\neq\tau(v)}}\geq\delta n\qquad\mbox{for some }\delta=\delta(k,d)>0.$$
In other words, the clusters are well-separated.
Furthermore, a ``typical'' cluster $\cC_i$ is characterized by a set of $\Omega(n)$ ``frozen'' vertices,
which have the same color in all colorings $\sigma\in\cC_i$.
Roughly speaking, a cover is a representation of a cluster $\cC_i$:
the cover details the colors of all the frozen vertices, while the non-frozen ones 
are represented by the ``joker color'' $0$.
We will define covers precisely in \Sec~\ref{Sec_Covers}.

The key idea behind the proof of \Thm~\ref{Thm_main} is to apply the first moment method to the number of covers.
Since, according to the cavity method, covers are in one-to-one correspondence with clusters,
we carry effectively out a first moment argument for the number of {\em clusters}.
The improvement over the ``classical'' first moment bound for the number of $k$-colorings results because
this approach allows us to completely ignore the cluster sizes $|\cC_i|$.
Indeed, close to the $k$-colorability threshold the cluster sizes are conjectured to vary wildly, as has  in part been
established rigorously in~\cite{ACOVilenchik}.
By contrast, the ``classical'' first moment argument amounts to putting a rather generous uniform bound on all the cluster sizes.

The ``freezing'' of vertices in $k$-colorings of $\gnm$ has been studied previously~\cite{Barriers}.
Formally, let us call a set $F$ of vertices {\bf\em $\delta$-frozen} in a $k$-coloring $\sigma$ of $\gnm$ 
if any other $k$-coloring $\tau$ such that $\tau(v)\neq\sigma(v)$ for some vertex $v\in F$ indeed satisfies
	$\abs{\cbc{v\in F:\sigma(v)\neq\tau(v)}}\geq\delta n.$
There is an explicitly known sharp threshold $\dfreeze=(1+o_k(1))k\ln k$, about half of $\dk$, such that for $d>\dfreeze$ \whp\ a  random $k$-coloring of $\gnm$ has
$\Omega(n)$ frozen vertices~\cite{Molloy}.
The threshold $\dfreeze$ coincides asymptotically with the largest average degree for which efficient algorithms are known to
find a $k$-coloring of $\gnm$ \whp~\cite{AchMolloy,GMcD}.
In fact, it has been hypothesized that the emergence of frozen vertices causes the failure of a wide class of ``local search'' algorithms~\cite{Barriers,Molloy}.

Yet the known results~\cite{Barriers,Molloy} on the freezing phenomenon only show that a {\em random} $k$-coloring of $\gnm$ ``freezes''.
It is not apparent that this poses an obstacle if we merely aim to find {\em  some} $k$-coloring.
As an important part of the proof of \Thm~\ref{Thm_main},
 we show that for $d$ close to $\dk$ (but strictly below the lower bound~(\ref{eqSecond})),
 in fact {\em all} $k$-colorings of $\gnm$ belong to a cluster with many frozen vertices \whp\

\begin{corollary}\label{Cor_allFrozen}
Assume that $d\geq 2k\ln k-\ln k-4+o_k(1)$.
There is a number $\delta_k>0$ such that \whp\ every $k$-coloring $\sigma$ of the random graph $\gnm$ 
has a set $F(\sigma)$ of $\delta_k$-frozen vertices of size $|F(\sigma)|\geq(1-o_k(1))n$.
\end{corollary}

Due to the (conjectured) relationship between freezing and the demise of local-search algorithms,
it would be interesting to identify the precise threshold where {\em all} the $k$-colorings of $\gnm$ are frozen.

\paragraph{Further related work.}

The problem of coloring $\gnm$ has been studied intensively over the past few decades.
Improving a prior result by Matula~\cite{Matula},
\Bollobas~\cite{BBColor} determined the asymptotic value of the chromatic number of dense random graphs.
\Luczak\ extended this result to sparse random graphs~\cite{LuczakColor}.
In the case that $d$ remains fixed as $n\ra\infty$, his result yields $\dk=(2+o_k(1))k\ln k$.
As mentioned above, Achlioptas and Naor~\cite{AchNaor} improved this result by obtaining the lower bound~(\ref{eqAN}).
In addition, \Luczak's result was sharpened in~\cite{Angelika} for $m\ll n^{5/4}$.

The problem of locating the threshold for $3$-colorability has received considerable attention as well.
The best current lower bound is $4.03$~\cite{AMo}.
Moreover, Dubois and Mandler~\cite{DuboisMandler} proved that $d_{3-\mathrm{col}}\leq4.9364$.
This improved over a stream of prior results~\cite{AchlioptasMolloy,DunneZito,Fountoulakis,Giotis,KKS}.

The key idea in this line of work is to estimate the first moment of the number of ``rigid''
colorings: for any two colors $1\leq i<j\leq k$, every vertex of color $i$ must have neighbors of color $j$~\cite{AchlioptasMolloy}.
Clearly, any $k$-colorable graph must have a rigid $k$-coloring.
At the same time, the number of rigid $k$-colorings can be expected
to be significantly smaller than the total number of $k$-colorings, and thus one might expect an improved first-moment upper bound.
However, in terms of the clustering scenario put forward by physicists, it is conceivable that many clusters contain
a large (in fact, exponentially large) number of rigid $k$-colorings.
Therefore, the idea of counting rigid $k$-colorings seems conceptually weaker than the approach of counting clusters
pursued in the present work.
In fact, the improvement obtained by counting rigid colorings appears to diminish for larger $k$~\cite{AchlioptasMolloy}.

A fairly new approach to obtaining upper bounds on thresholds in random constraint satisfaction problems is the
use of the {\em interpolation method}~\cite{BGT10,FranzLeone,Guerra,Panchenko}.
This technique gives an upper bound on, e.g., the $k$-colorability threshold in terms of a variational problem that is related to the
statistical mechanics techniques.
However, this variational problem appears to be difficult to solve.
Thus, it is not clear (to me) how an {\em explicit} upper bound as stated in \Thm~\ref{Thm_main}
can be obtained from  the interpolation method.

Dani, Moore and Olson~\cite{Dani} studied a variant of the graph coloring problem in which each pair of $(u,v)$ of vertices comes with
a random permutation $\pi_{u,v}$ of the $k$ possible colors;
this gives rise to a concept of ``permuted'' $k$-colorings.
They obtained an upper bound of $2k\ln k-\ln k-1+o_k(1)$ on the threshold for the existence of permuted $k$-colorings.
The proof is based on counting the total weight of $k$-colorings and using an isoperimetric inequality.
Moreover, as pointed out in~\cite{Dani}, physics intuition suggests that the threshold in the permuted $k$-coloring problem matches the
``unpermuted'' $k$-colorability threshold.

In the context of satisfiability, Maneva and Sinclair~\cite{Maneva} used the concept of covers to obtain a conditional upper bound on the random $3$-SAT threshold.
Roughly speaking, the condition that they need is that \whp\ all satisfying assignments have frozen variables.
However, verifying this condition in random 3-SAT is an open problem.
(That said, it is conceivable that the approach used in~\cite{Maneva} might yield a better upper bound on the $k$-SAT threshold for large $k$.)

\section{Preliminaries}

Let $\brk k=\cbc{1,2,\ldots,k}$.
Because \Thm~\ref{Thm_main} and \Cor~\ref{Cor_allFrozen} are asymptotic statements in both $k$ and $n$, we may
generally assume that $k\geq k_0$ and $n\geq n_0$, where $k_0,n_0$ are constants that are chosen sufficiently large
for the various estimates to hold.

We perform asymptotic considerations with respect to both $k$ and $n$.
When referring to asymptotics in $k$, we use the notation $O_k(\cdot)$, $o_k(\cdot)$, etc.
Asymptotics with respect to $n$ are just denoted by $O(\cdot)$, $o(\cdot)$, etc.

If $G$ is a (multi-)graph and $A,B$ are sets of vertices, then we let $e_G(A,B)$ denote the number of $A$-$B$-edges in $G$.
Moreover, $e_G(A)$ denotes the number of edges inside of $A$.
If $A=\cbc v$ is a singleton, we just write $e_G(v,B)$.
The reference to $G$ is omitted where it is clear from the context.

\paragraph{Working with independent edges.}

The random graph $\gnm$ consists of $m$ edges that are chosen \emph{almost} independently.
To simplify some of the arguments below, we are going to work with a random multi-graph model $G'(n,m)$ in which
edges are perfectly independent.
More precisely, $G'(n,m)$ is obtained as follows:
	let $\vec e=(\vec e_1,\ldots,\vec e_m)\in(V\times V)^m$ be a uniformly random $m$-tuple of ordered pairs of vertices.
In other words, each $\vec e_i$ is chosen uniformly out of all $n^2$ possible vertex pairs, independently of all the others.
Now, let $G'(n,m)$ be the random multi-graph comprising of $\vec e_1,\ldots,\vec e_m$ viewed as undirected edges.
Thus, $G'(n,m)$ may have self-loops (if $\vec e_i=(v,v)$ for some index $i$) as well as multiple edges (if, for example, $\vec e_i=(u,v)$ and $\vec e_j=(v,u)$
with $1\leq i<j\leq m$ and $u\neq v$).
The two random graph models are related as follows.

\begin{lemma}\label{Lemma_independentEdges}
For any event $\cA$ we have
	$\pr\brk{\gnm\in\cA}\leq O(1)\cdot\pr\brk{G'(n,m)\in\cA}$.
\end{lemma}
\begin{proof}
The random graph $G'(n,m)$ has \emph{at most} $m$ distinct edges, and no self-loops.
Let $\cE$ be the event that it has \emph{exactly} $m$ edges.
This is the case iff $\vec e_1,\ldots,\vec e_m$ induce pairwise distinct undirected edges.
Given the event $\cE$, $G'(n,m)$ is identical to $G(n,m)$.
Hence,
	\begin{equation}\label{eqindependence1}
	\pr\brk{\gnm\in\cA}=\pr\brk{G'(n,m)\in\cA|\cE}\leq\pr\brk{G'(n,m)\in\cA}/\pr\brk\cE
	\end{equation}
Now,
	\begin{eqnarray*}
	\pr\brk\cE&\geq&\prod_{i=0}^{m}\bc{1-\frac{2i+n}{n^2}}=\exp\brk{\sum_{i=0}^{m-1}\ln(1-(2i+n)/n^2)}\geq\exp(-d-2d^2)=\Omega(1).
	\end{eqnarray*}
Thus, the assertion follows from~(\ref{eqindependence1}).
\qed\end{proof}

\paragraph{The Chernoff bound.}

We need the following Chernoff bound on the tails of a binomially distributed random variable (e.g.,~\cite[p.~21]{JLR}).

\begin{lemma}\label{Lemma_Chernoff}
Let $\varphi(x)=(1+x)\ln(1+x)-x$.
Let $X$ be a binomial random variable with mean $\mu>0$.
Then for any $t>0$ we have
	\begin{eqnarray*}
	\pr\brk{X>\Erw\brk X+t}&\leq&\exp(-\mu\cdot\varphi(t/\mu)),\\
	\pr\brk{X<\Erw\brk X-t}&\leq&\exp(-\mu\cdot\varphi(-t/\mu)).
	\end{eqnarray*}
In particular, for any $t>1$ we have
	$\pr\brk{X>t\mu}\leq\exp\brk{-t\mu\ln(t/\eul)}.$
\end{lemma}

\paragraph{Balls and bins.}
Consider a balls and bins experiment where $\mu$ balls
are thrown independently and uniformly at random into $\nu$ bins.
Thus, the probability of each distribution of balls into bins equals $\nu^{-\mu}$.
We will need the following well-known ``Poissonization lemma'' (e.g., \cite[Section~2.6]{Durrett}).

\begin{lemma}\label{Lemma_BallsAndBins}
In the above experiment let $e_i$ be the number of balls in bin $i\in\brk\nu$.
Moreover, let $\lambda>0$ and let $(b_i)_{i\in\brk\nu}$ be a family of independent Poisson variables, each with mean $\lambda$.
Then for any sequence $(t_i)_{i\in\brk\nu}$ of non-negative integers such that $\sum_{i=1}^\nu t_i=\mu$ we have
	$$\pr\brk{\forall i\in\brk\nu:e_i=t_i}=\pr\bigg[\forall i\in\brk\nu:b_i=t_i\,\bigg|\,\sum_{i=1}^\nu b_i=\mu\bigg].$$
Hence, the joint distribution of $(e_i)_{i\in\brk\nu}$ coincides with the joint distribution of $(b_i)_{i\in\brk\nu}$
given $\sum_{i=1}^\nu b_i=\mu$.
\end{lemma}

We are typically going to use \Lem~\ref{Lemma_BallsAndBins} to obtain an {\em upper} bound on the probability on the left hand side.
Therefore, the following simple corollary will come in handy.

\begin{corollary}\label{Cor_BallsAndBins}
With the notation of \Lem~\ref{Lemma_BallsAndBins}, assume that $\lambda=\mu/\nu>0$.
Then for any sequence $(t_i)_{i\in\brk\nu}$ of non-negative integers such that $\sum_{i=1}^\nu t_i=\mu$ we have
	$$\pr\brk{\forall i\in\brk\nu:e_i=t_i}\leq O(\sqrt\mu)\cdot \pr\brk{\forall i\in\brk\nu:b_i=t_i}.$$
\end{corollary}
\begin{proof}
Let $b=\sum_{i=1}^\nu b_i=\mu$.
Since the $b_i$ are independent Poisson variables with means $\lambda=\mu/\nu$, $b$ is Poisson with mean $\mu$.
By Stirling's formula,
	$\pr\brk{b=\mu}=\mu^\mu\exp(-\mu)/\mu!=\Omega(\mu^{-1/2})$.
Hence, \Lem~\ref{Lemma_BallsAndBins} yields
	\begin{eqnarray*}
	\pr\brk{\forall i\in\brk\nu:e_i=t_i}&=&\frac{\pr\brk{\forall i\in\brk\nu:b_i=t_i}}{\pr\brk{b=\mu}}
		=O(\sqrt\mu)\cdot\pr\brk{\forall i\in\brk\nu:b_i=t_i},
	\end{eqnarray*}
as claimed.
\qed\end{proof}

\section{Covers}\label{Sec_Covers}

Let $G=(V,E)$ be a graph, let $k\geq k_0$ be an integer, and let $\sigma:V\ra\brk k$ be a $k$-coloring of $G$.
We would like to identify a set $F\subset V$ of vertices whose colors cannot be changed easily by a ``local'' recoloring of a few vertices.
For instance, if $v$ is a vertex that does not have a neighbor of color $j$ for some $j\in\brk k\setminus\cbc{\sigma(v)}$, then
$v$ can be recolored easily.
More generally, we would like to say that, recursively, a vertex can be recolored easily if there is a color $j$ such that all its neighbors of color $j$
can be easily recolored.
To formalize this, we need the following concept.

\begin{definition}
Let $\zeta:V\ra\cbc{0,1,\ldots,k}$.
We call $v\in V$ {\bf\em stable} under $\zeta$ if $\zeta(v)\neq0$ and if for any color $j\in\brk k\setminus\cbc{\zeta(v)}$ there are
at least two neighbors $u_1,u_2$ of $v$ such that $\zeta(u_1)=\zeta(u_2)=j$.
\end{definition}

Now, consider the following 
{\bf\em whitening process} that, given a $k$-coloring $\sigma$ of $G$, returns a map $\hat\sigma:V\ra\cbc{0,1,\ldots,k}$;
the idea is that $\hat\sigma(v)=0$ for all $v$ that are easy to recolor.
\begin{description}
\item[WH1.] Initially, let $\hat\sigma(v)=\sigma(v)$ for all $v\in V$.
\item[WH2.] While there exist a vertex $v\in V$ with $\hat\sigma(v)\neq0$ that is not stable under $\hat\sigma$, set $\hat\sigma(v)=0$.
\end{description}
The process {\bf WH1--WH2} is similar to processes studied in~\cite{ParisiWhitening,LenkaThesis} in the context of random graph coloring,
and in~\cite{AchlioptasRicci} in the context of random $k$-SAT.
(The term ``whitening process'' stems from~\cite{ParisiWhitening}.)
Clearly, the final outcome $\hat\sigma$ of the whitening process is independent of the order in which {\bf WH2} proceeds.

The intuition behind the whitening process is that if we attempt to recolor some stable vertex $v$
with another color $j\in\brk k\setminus\cbc{\hat\sigma(v)}$, then we will have
to recolor {\em two} additional stable vertices $u_1,u_2$ as well.
Hence, any attempt to recolor a stable vertex is liable to trigger an avalanche of further recolorings (unless the graph $G$
has an abundance of short cycles, which is well-known not to be the case in the random graph $\gnm$ \whp).

The following definition is going to lead to a neat description of the outcome of the whitening process.

\begin{definition}
A {\bf \em$k$-cover} in $G$ is a map $\zeta:V\ra\cbc{0,1,\ldots,k}$ with the following properties.
\begin{description}
\item[CV1.] There is no edge $e=\cbc{u,v}$ such that $\zeta(u)=\zeta(v)\neq0$.
\item[CV2.] If $\zeta(v)\neq0$, then $v$ is stable under $\zeta$.
\item[CV3.] If $\zeta(v)=0$, then there are $i,j\in\brk k$, $i\neq j$, such that
			$v$ does not have a neighbor $u$ with $\zeta(u)=i$ and $v$ has at most
			one neighbor $w$ with $\zeta(w)=j$.
\end{description}
\end{definition}
The concept of covers is very closely related and, in fact, inspired by the properties of certain
fixed points of the Survey Propagation message passing procedure~\cite{MM}.
(To my knowledge, the term ``cover'' has not been used previously in the context of $k$-colorability, although it appears
to be in common use in the context of satisfiability~\cite{Maneva}.)

Now, the outcome of $\hat\sigma$ is the cover characterized by the following two properties.
\begin{enumerate}
\item[i.] For all vertices $v$ such that $\hat\sigma(v)\neq0$ we have $\hat\sigma(v)=\sigma(v)$.
\item[ii.] Subject to i., $|\hat\sigma^{-1}(0)|$ is minimum.
\end{enumerate}
Of course, in general the graph $G$ may have many $k$-covers that cannot be obtained
from a $k$-coloring via the whitening process.
This motivates

\begin{definition}
A $k$-cover $\zeta$ of $G$ is {\bf\em valid} if $G$ has a $k$-coloring $\sigma$ such that $\zeta=\hat\sigma$.
\end{definition}

To prove \Thm~\ref{Thm_main}, we perform a first moment argument for the number of valid $k$-covers.
The main task is to show that the all-$0$ cover (i.e., $\zeta(v)=0$ for all vertices $v$) is not a valid $k$-cover in $\gnm$ \whp\
To this end, we need to establish a few basic properties that all $k$-colorings of $\gnm$ have \whp\
More precisely, in \Sec~\ref{Sec_Prop_firstMoment} we are going to prove the following via a ``standard'' first moment argument
over $k$-colorings.

\begin{proposition}\label{Prop_firstMoment}
Assume that $k\geq k_0$ for a sufficiently large constant $k_0$.
Moreover, assume that $d=2k\ln k-\ln k-c$, with $0\leq c\leq 4$.
\begin{enumerate}
\item Let $Z$ be the number of $k$-colorings of $\gnm$. Then
		$\frac1n\ln\Erw\brk Z=\frac{c+o_k(1)}{2k}.$
\item \Whp\ all $k$-colorings of $\gnm$ satisfy $|\sigma^{-1}(i)|=(1+o_k(1))\frac nk$ for all $i\in\brk k$.
\item In fact,
	\whp\ $\gnm$ does not have a $k$-coloring $\sigma$ such that $|\sigma^{-1}(i)-n/k|>n/(k\ln^4k)$ for more than $\ln^8k$ colors $i\in\brk k$.
\end{enumerate}
\end{proposition}

\noindent
Building upon \Prop~\ref{Prop_firstMoment}, we will establish the following properties of valid $k$-covers  in \Sec~\ref{Sec_Prop_core}.

\begin{proposition}\label{Prop_core}
There is a number $k_0$ such that for 
$k\geq k_0$ and $2k\ln k-\ln k-4\leq d\leq2k\ln k$ any valid $k$-cover $\zeta$ of $\gnm$ has the following properties \whp
\begin{enumerate}
\item We have $|\zeta^{-1}(0)|\leq nk^{-2/3}$.
\item For all $i\in\brk k$ we have $|\zeta^{-1}(i)|=(1+o_k(1))n/k$.
\item In fact, there are no more than $\ln^9k$ indices $i\in\brk k$ such that $|\zeta^{-1}(i)-n/k|>n/(k\ln^3k)$.
\end{enumerate}
\end{proposition}

\noindent
Finally, in \Sec~\ref{Sec_Prop_noCover} we perform the first moment argument over $k$-covers.

\begin{proposition}\label{Prop_noCover}
There is  $\eps_k=o_k(1)$ such that
for $d\geq2k\ln k-\ln k-1+\eps_k$ \whp\ the random graph $\gnm$
does not have a $k$-cover with properties 1.--3.\ from \Prop~\ref{Prop_core}.
\end{proposition}

\noindent
\Thm~\ref{Thm_main} is immediate from \Prop s~\ref{Prop_core} and~\ref{Prop_noCover}.
Furthermore, we will prove \Cor~\ref{Cor_allFrozen} in \Sec~\ref{Sec_Prop_core}.

\section{Proof of \Prop~\ref{Prop_firstMoment}}\label{Sec_Prop_firstMoment}

The proof of \Prop~\ref{Prop_firstMoment} is very much  based on standard arguments, reminiscent but unfortunately not (quite) identical to
estimates from, e.g., \cite{AchlioptasMolloy}.
Suppose $d=2k\ln k-\ln k-c$ with $0\leq c\leq 4$.
Throughout this section we work with the random graph $\Gnm$ with $m$ independent edges.

\begin{lemma}\label{Lemma_firstM}
Let $\nu=(\nu_1,\ldots,\nu_k)$ be a $k$-tuple of non-negative integers such that $\sum_{i=1}^k \nu_i=n$.
Let $Z_\nu$ be the number of $k$-colorings $\sigma$ of $\Gnm$ such that
$\abs{\sigma^{-1}(i)}=\nu_i$ for all $i\in\brk k$.
Then
	\begin{equation}\label{eqLemma_firstM}
	\ln\Erw\brk{Z_\nu}=o(n)+\sum_{i=1}^k\nu_i\ln(n/\nu_i)+\frac d2\ln\brk{1-\sum_{i=1}^k\bcfr{\nu_i}{n}^2}.
	\end{equation}
\end{lemma}
\begin{proof}
Let $\Sigma_\nu$ be the set of all $\sigma:V\ra\brk k$ such that $\abs{\sigma^{-1}(i)}=\nu_i$ for all $i\in\brk k$.
By Stirling's formula,
	\begin{eqnarray}\label{eqFirstM0}
	\ln\abs{\Sigma_\nu}&=&o(n)+\sum_{i=1}^k\nu_i\ln(n/\nu_i).
	\end{eqnarray}
Furthermore, the probability of being a $k$-coloring in $\Gnm$ is the same for all $\sigma\in\Sigma_\nu$.
In fact, due to the independence of the edges in $\Gnm$, this probability is
	$q=(1-\sum_{i=1}^k(\nu_i/n)^2)^m,$
because $\sigma$ is a $k$-coloring iff each of the color classes $\sigma^{-1}(i)$ is an independent set.
As $\Erw\brk{Z_\nu}=\abs{\Sigma_\nu}\cdot q$,
the assertion follows from (\ref{eqFirstM0}).
\qed\end{proof}

\begin{corollary}\label{Cor_firstM}
Let $Z$ be the total number of $k$-colorings of $\Gnm$.
We have
	$$\frac1n\ln\Erw\brk Z=\ln k+\frac d2\ln(1-1/k)+o(1)=\frac{c}{2k}+O_k(\ln k/k^2).$$
\end{corollary}
\begin{proof}
Let $\cN$ be the set of all $k$-tuples $\nu=(\nu_1,\ldots,\nu_k)$ of non-negative integers
such that $\sum_{i=1}^k\nu_i=n$.
Then
	$\Erw\brk{Z}=\sum_{\nu\in\cN}\Erw\brk{Z_\nu}\leq n^k\max_{\nu\in\cN}\Erw\brk{Z_\nu}.$
Hence, \Lem~\ref{Lemma_firstM} yields
	\begin{eqnarray}
	\frac1n\ln\Erw\brk{Z}
		&=&o(1)+\max\cbc{\sum_{i=1}^k\nu_i\ln(n/\nu_i)+\frac d2\ln\brk{1-\sum_{i=1}^k\bcfr{\nu_i}{n}^2}:\nu\in\cN}.
			\label{eqCorFirstM}
	\end{eqnarray}
Letting $\cA$ be the set of all $k$-tuples $\alpha=(\alpha_1,\ldots,\alpha_k)\in[0,1]^k$ such that
$\sum_{i=1}^k\alpha_i=1$, we obtain from~(\ref{eqCorFirstM})
	\begin{equation}			\label{eqCorFirstM2}
	\frac1n\ln\Erw\brk{Z}=o(1)+
		\max\cbc{-\sum_{i=1}^k\alpha_i\ln(\alpha_i)+\frac d2\ln\brk{1-\sum_{i=1}^k\alpha_i^2}:\alpha\in\cA}.
	\end{equation}
The entropy function $-\sum_{i=1}^k\alpha_i\ln(\alpha_i)$ is well-known to attain its maximum 
at the point $\alpha=\frac1k\vecone$ with all $k$ entries equal to $1/k$.
Furthermore, the sum of squares $\sum_{i=1}^k\alpha_i^2$ attains its minimum
at $\alpha=\frac1k\vecone$ as well.
Hence, the term
	$\frac d2\ln[1-\sum_{i=1}^k\alpha_i^2]$, and thus~(\ref{eqCorFirstM2}), is maximized at $\frac1k\vecone$.
Consequently,
	\begin{eqnarray*}
	\frac1n\ln\Erw\brk Z&=&\ln k+\frac d2\ln(1-1/k)+o(1)
		=\ln k-\frac d2\brk{\frac 1k+\frac1{2k^2}+O_k(k^{-3})}\\
		&=&\ln k-\brk{k\ln k-\frac{\ln k}{2}-\frac c{2}}\cdot\brk{\frac 1k+\frac1{2k^2}+O_k(k^{-3})}
		=\frac c{2k}+O(\ln k/k^2),
	\end{eqnarray*}
as claimed.
\qed\end{proof}

\begin{corollary}\label{Cor_equi}
\Whp\ all $k$-colorings $\sigma$ of $\Gnm$
satisfy $|\sigma^{-1}(i)|=(1+o_k(1))\frac nk$ for all $i\in\brk k$, and
there is no $k$-coloring $\sigma$ such that $|\sigma^{-1}(i)-1/k|>1/(k\ln^4k)$ for more than $\ln^8k$ colors $i\in\brk k$.
\end{corollary}
\begin{proof}
Let $\nu=(\nu_1,\ldots,\nu_k)$ be a $k$-tuple of non-negative integers such that $\sum_{i=1}^k\nu_i=n$.
We are going to estimate $\Erw[Z_\nu]$ in terms of how much $\nu$ deviates from the ``flat'' vector $(n/k,\ldots,n/k)$.
To this end, we  compute the first two differentials of~(\ref{eqLemma_firstM}).
Set $\alpha=(\alpha_1,\ldots,\alpha_k)=\nu/n$ and let
	$f(\alpha)=-\sum_{i=1}^k\alpha_i\ln\alpha_i+\frac d2\ln(1-\sum_{i=1}^k\alpha_i^2).$
Since $\sum_{i=1}^k\alpha_i=1$, we can eliminate the variable $\alpha_k=1-\sum_{i=1}^{k-1}\alpha_i$.
Hence, we obtain for $i,j\in\brk{k-1}$, $i\neq j$
	\begin{eqnarray}\nonumber
	\frac{\partial f}{\partial\alpha_i}&=&\ln(\alpha_k/\alpha_i)+\frac{d(\alpha_k-\alpha_i)}{1-\norm\alpha_2^2},\\
	\frac{\partial^2 f}{\partial\alpha_i^2}&=&-\frac1{\alpha_k}-\frac1{\alpha_i}-\frac{2d}{1-\norm\alpha_2^2}-\frac{2d(\alpha_k-\alpha_i)^2}{(1-\norm\alpha_2^2)^2},
		\label{eqHessian1}\\
	\frac{\partial^2 f}{\partial\alpha_i\partial\alpha_j}&=&-\frac1{\alpha_k}-\frac d{1-\norm\alpha_2^2}-\frac{2d(\alpha_k-\alpha_i)(\alpha_k-\alpha_j)}{(1-\norm\alpha_2^2)^2}
		\label{eqHessian2}.
	\end{eqnarray}
In particular, the first differential $Df$  vanishes at $\alpha=\frac1k\vecone$.
At this point, the Hessian $D^2f=(\frac{\partial^2f}{\partial\alpha_i\partial\alpha_j})_{i,j\in\brk{k-1}}$ is negative-definite, whence $\alpha=\frac1k\vecone$ is a local maximum.
Because the rank-one matrix $((\alpha_k-\alpha_i)\cdot(\alpha_k-\alpha_j))_{i,j\in\brk{k-1}}$ is positive semidefinite for all $\alpha$,
(\ref{eqHessian1}) and~(\ref{eqHessian2}) show that $D^2f$ is negative-definite for all $\alpha$.
In fact, due to the $-\frac{2d}{1-\norm\alpha_2^2}$ term in~(\ref{eqHessian1}), all its eigenvalues are smaller than $-\frac{d}{1-\norm\alpha_2^2}\leq-d$.
Therefore, Taylor's theorem yields that
	$$f(\alpha)\leq f(k^{-1}\vecone)-\frac d2\norm{\alpha-k^{-1}\vecone}_2^2$$
for all $\alpha$.
Hence, \Cor~\ref{Cor_firstM} implies that
	\begin{equation}\label{eqTaylorBound}
	\frac1n\ln\Erw[Z_\nu]\leq \frac{c}{2k}+O_k(\ln k/k^2)-\frac d2\norm{\alpha-k^{-1}\vecone}_2^2.
	\end{equation}
Since $d=(2-o_k(1))k\ln k$, the right hand side of~(\ref{eqTaylorBound}) is negative if either
\begin{itemize}
\item $\max_{i\in\brk k}|\alpha_i-k^{-1}|>(k\ln^{1/3}k)^{-1}$, or
\item there are more than $\ln^8k$ indices $i\in\brk k$ such that $|\alpha_i-1/k|>(k\ln^4k)^{-1}$.
\end{itemize}
Thus, Markov's inequality shows that \whp\ there is no $k$-coloring with either of these properties.
\qed\end{proof}

\medskip\noindent
Finally, \Prop~\ref{Prop_firstMoment} is immediate from \Lem~\ref{Lemma_independentEdges} and Corollaries~\ref{Cor_firstM} and~\ref{Cor_equi}.

\section{Proof of \Prop~\ref{Prop_core}}\label{Sec_Prop_core}

Suppose $d=2k\ln k-\ln k-c$ with $0\leq c\leq 4$.
Throughout this section we work with the random graph $\Gnm$ with $m$ independent edges.

\subsection{The core}

Let $\sigma:V\ra\brk k$ be a map such that $|\sigma^{-1}(i)|=(1+o_k(1))\frac nk$ for all $i\in\brk k$.
Moreover, let $G'(\sigma)$ be the random multi-graph $\Gnm$ conditional on $\sigma$ being a valid $k$-coloring.
Thus, $G'(\sigma)$ consists of $m$ independent random edges $\vec e_1=(\vec u_1,\vec v_1),\ldots,\vec e_m=(\vec u_m,\vec v_m)$
such that $\sigma(\vec u_i)\neq\sigma(\vec v_i)$ for all $i\in\brk m$.
To prove \Prop~\ref{Prop_core} we need to show that with a very high probability, a large number of vertices of $G'(\sigma)$
will remain ``unscathed'' by the whitening process {\bf WH1--WH2}.
To exhibit such vertices, we consider the following construction.
Let $\ell=\exp(-7)\ln k$ and assume that $k\geq k_0$ is large enough so that $\ell>3$.
Let $V_i=\sigma^{-1}(i)$ for $i\in\brk k$.
\begin{description}
  \item[CR1] For $i\in\brk k$ let $W_i=\cbc{v\in V_i:\exists j\neq i:e(v,V_j)<3\ell}$ and $W=\bigcup_{i=1}^k W_i$.
  \item[CR2] Let $U=\cbc{v\in V:\exists j:e(v,W_j)>\ell}$.
  \item[CR3] Set $Y=U$. While there is a vertex $v \in V \setminus Y$ that has $\ell$ or more neighbors in $Y$, add $v$ to $Y$.
\end{description}
We call the graph $G'(\sigma)-W-Y$ obtained by removing the vertices in $W\cup Y$ the {\bf\em core} of $G'(\sigma)$.
By construction, every vertex $v$ in the core has at least $\ell$ neighbors of each color $j\neq\sigma(v)$ that also belong to the core.
In effect, if $\hat\sigma$ is the outcome of the whitening process applied to $G'(\sigma)$, then
$\hat\sigma(v)=\sigma(v)$ for all vertices $v$ in the core.

The construction {\bf CR1--CR3} has been considered previously to show that a {\em random} $k$-coloring of the random graph $\gnm$
has many frozen vertices \whp~\cite{Barriers,ACOVilenchik}.
In the present context we need to perform a rather more thorough analysis of the process {\bf CR1--CR3} to show that \whp\
{\em all} $k$-colorings $\sigma$ of  $\gnm$ induce a non-zero cover $\hat\sigma$.
To obtain such a strong result, we need to control the large deviations of various quantities, particularly the sizes of the sets $W$, $W_i$ and $U$.
More precisely, in \Sec~\ref{Sec_Lem_size_of_W} we prove

\begin{lemma}\label{Lem_size_of_W}
With probability at least $1-\exp(-16n/k)$ the random graph $G'(\sigma)$ has the following properties.
\begin{enumerate}
\item We have $|W|\leq nk^{-0.7}$.
\item For all $i\in\brk k$ we have $|W_i|\leq\frac{n\ln\ln k}{k\ln k}$.
\item There are no more than $\ln^4k$ indices $i\in\brk k$ such that $|W_i|\geq \frac{n}{k\ln^4 k}$.
\end{enumerate}
\end{lemma}
Moreover, in \Sec~\ref{Sec_Lem_size_of_U} we are going to establish

\begin{lemma}\label{Lem_size_of_U}
In $G'(\sigma)$ we have
	$\pr\brk{\abs U>\frac{n\ln\ln k}{k\ln k}}\leq\exp(-10n/k).$
\end{lemma}

\noindent
To estimate the size of $Y$ we use the following observation.

\begin{lemma}\label{Lemma_expansion}
\Whp\ the random graph $\Gnm$ has the following property.
\begin{equation}\label{eqLemma_expansion}
\parbox[c]{12cm}{For any set $\cY\subset V$ of size $|\cY|\leq\lceil\frac{2n\ln\ln k}{k\ln k}\rceil$ we have
	$e(\cY)<\frac{\ell}2|\cY|$}
\end{equation}
\end{lemma}
\begin{proof}
For any fixed set $\cY$ of size $0<yn\leq \lceil\frac{2n\ln\ln k}{k\ln k}\rceil$ the number $e(\cY)$ of edges spanned by $\cY$ in $\Gnm$
is binomially distributed  with mean 
	$$\Erw\brk{e(\cY)}=\frac{m(yn)^2}{n^2}=(1+o_k(1))y^2dn/2\leq 2y^2nk\ln k$$
Hence, by the Chernoff bound 
	\begin{equation}\label{eqLem_size_of_Z2a}
	\pr\brk{e(\cY)\geq y\ell n/2}
		\leq\exp\brk{\frac{y\ell n}2\ln\bcfr{y\ell n/2}{\eul\cdot\Erw\brk{e(\cY)}}}\leq\exp\brk{\frac{y\ell n}{3}\ln(ky)}.
	\end{equation}
Since $ky\leq 3\ln\ln k/\ln k$ and $\ell=\Omega_k(\ln k)$,  (\ref{eqLem_size_of_Z2a}) yields
	\begin{equation}\label{eqLem_size_of_Z2}
	\pr\brk{e(\cY)\geq y\ell n/2}\leq\exp\brk{3yn\ln(y)}.
	\end{equation}
Further, by Stirling's formula the total number of sets $\cY\subset V$ of size $yn$ is
	\begin{equation}\label{eqLem_size_of_Z3}
	\bink n{yn}\leq\exp\brk{yn(1-\ln y)}\leq\exp(-2yn\ln y).
	\end{equation}
Combining (\ref{eqLem_size_of_Z2}) and (\ref{eqLem_size_of_Z3}) with the union bound, we obtain
	$$\pr\brk{\exists \cY\subset V:|\cY|=yn,\,e(\cY)\geq \ell\abs\cY/2}\leq\exp\brk{y n\ln(y)}.$$
Taking the union bound over all possible sizes $yn$ completes the proof.
\qed\end{proof}

\medskip
\noindent{\em Proof of Proportion \ref{Prop_core}.} 
By \Prop~\ref{Prop_firstMoment} \whp\ all $k$-colorings $\sigma$ of the random graph $\Gnm$ satisfy
$|\sigma^{-1}(i)|=(1+o_k(1))n/k$ for all $i\in\brk k$.
Let us call such a $k$-coloring $\sigma$ of $G=\Gnm$ 
{\em good} if it has the following two properties (and {\em bad} otherwise):
\begin{description}
\item[G1.] Step {\bf CR1} applied to $G,\sigma$ yields sets $W_1,\ldots,W_k,W$ that satisfy the three properties in \Lem~\ref{Lem_size_of_W}.
\item[G2.] The set $U$ created in step {\bf CR2} has size $\abs U>\frac{n\ln\ln k}{k\ln k}$.
\end{description}
Let $Z_{\mathrm{bad}}$ be the number of bad $k$-colorings of $\Gnm$.
Since $G'(\sigma)$ is just the random graph $\Gnm$ conditional on $\sigma$ being a $k$-coloring, we have
	\begin{eqnarray*}
	\Erw\brk{Z_{\mathrm{bad}}}&=&\sum_{\sigma}\pr\brk{\sigma\mbox{ is a $k$-coloring of $\gnm$}}\cdot\pr\brk{\sigma\mbox{ is bad in }\Gnm|
			\sigma\mbox{ is a $k$-coloring}}\\
		&=&\sum_{\sigma}\pr\brk{\sigma\mbox{ is a $k$-coloring of $\gnm$}}\cdot\pr\brk{\sigma\mbox{ is bad in }G'(\sigma)}\\
		&\leq&2\exp(-10n/k)\cdot\sum_{\sigma}\pr\brk{\sigma\mbox{ is a $k$-coloring of $\gnm$}}\qquad\mbox{[by \Lem s~\ref{Lem_size_of_W} and~\ref{Lem_size_of_U}]}\\
		&\leq&2\exp(-10n/k)\cdot\exp(cn/k+o(n))\qquad\qquad\qquad\qquad\quad\mbox{ [by \Prop~\ref{Prop_firstMoment}]}\\
		&\leq&\exp(-6n/k+o(n))=o(1).\qquad\qquad\qquad\qquad\qquad\qquad\mbox{\,[as $c\leq 4$]}
	\end{eqnarray*}
Hence, \whp\ the random graph $\Gnm$ does not have a bad $k$-coloring.

Now, consider a good $k$-coloring $\sigma$.
By \Lem~\ref{Lemma_expansion}, we may assume that~(\ref{eqLemma_expansion}) holds.
To bound the size of the set $Y$ created by step {\bf CR3},
observe that each vertex that is added to $Y$ contributes $\ell$ extra edges to the subgraph spanned by $Y$.
Thus, assume that $|Y|>\frac{2n\ln\ln k}{k\ln k}$ and consider  the first time step {\bf CR3} has got a set $Y'$ 
of size  $\lceil \frac{2n\ln\ln k}{k\ln k}\rceil$.
Then $Y'$ spans at least $y\ell n/2$ edges, in contradiction to~(\ref{eqLemma_expansion}).
Hence, \whp\ $\Gnm$ is such that
	\begin{equation}\label{eqGoodY}
	\mbox{for {\em any} good $k$-coloring the set $Y$ constructed by {\bf CR3} has size at most $ \frac{2n\ln\ln k}{k\ln k}$.}
	\end{equation}

If~(\ref{eqGoodY}) is true and $\Gnm$ does not have a bad $k$-coloring, then for any $k$-coloring $\sigma$ the set $W\cup Y$ constructed by {\bf CR1--CR3}
has size at most $|W\cup Y|\leq k^{-0.7}n+ \frac{2n\ln\ln k}{k\ln k}\leq nk^{-2/3}$ (the bound on $|W|$ follows from {\bf G1}).
This shows the first property asserted in \Prop~\ref{Prop_core}, because
 the construction {\bf CR1--CR3} ensures that the
 cover $\hat\sigma$ obtained from $\sigma$ via the whitening process {\bf WH1--WH2} satisfies $\hat\sigma(v)=\sigma(v)$
for all $v\in V\setminus(W\cup Y)$.
By the same token, the second assertion follows because by {\bf G1} and~(\ref{eqGoodY}) for every color $i\in\brk k$ we have
	$$|\sigma^{-1}(i)\cap (W\cup Y)|\leq|W_i|+|Y|=n\cdot o_k(1/k).$$
Finally, the {\bf G1} and~(\ref{eqGoodY}) also imply that
there cannot be more than $\ln^9k$ indices $i\in\brk k$ such that $|\sigma^{-1}(i)-n/k|>n/(k\ln^3k)$,
which is the third assertion.
\qed

\medskip\noindent{\em Proof of \Cor~\ref{Cor_allFrozen}.}
We claim that the vertices in $F=V\setminus(W\cup Y)$ are $\delta$-frozen \whp\ for $\delta=1/(k\ln k)$.
Indeed, assume that $\tau$ is another $k$-coloring such that the set $\Delta=\cbc{v\in F:\tau(v)\neq\sigma(v)}$
has size
	$0<\abs\Delta<\delta n.$
Every vertex $v\in\Delta$ has at least $\ell$ neighbors in $\Delta$.
Indeed, the construction {\bf CR1--CR3} ensures that every vertex $v\in\Delta$ has at least $\ell$ neighbors colored $\tau(v)\neq\sigma(v)$ in $\sigma$.
Because $\tau$ is a $k$-coloring, all of these neighbors must belong to $\Delta$ as well.
Hence, $\Delta$ violates~(\ref{eqLemma_expansion}).
Thus, the assertion follows from \Lem~\ref{Lemma_expansion}.
\qed

\subsection{Proof of \Lem~\ref{Lem_size_of_W}}\label{Sec_Lem_size_of_W}

We begin by estimating the number of edges between different color classes.
Recall that $V_i=\sigma^{-1}(i)$ for $i\in\brk k$, and that we are assuming that $|V_i|=(1+o_k(1))n/k$.
Let $\nu_i=|V_i|$ for $i=1,\ldots,k$.

\begin{lemma}\label{Lemma_fixEdgeCounts}
In $G'(\sigma)$ we have
	\begin{eqnarray*}
	\pr\brk{\min_{1\leq i<j\leq k}e(V_i,V_j)\leq \frac{0.99dn}{k^2}}&\leq&\exp(-11 n/k)\quad\mbox{ and }\quad\\
	\pr\brk{\max_{1\leq i<j\leq k}e(V_i,V_j)\geq \frac{1.01dn}{k^2}}&\leq&\exp(-11 n/k).
	\end{eqnarray*}
\end{lemma}
\begin{proof}
Because the edges $\vec e_1,\ldots,\vec e_m$ are chosen independently,
for any pair $1\leq i<j\leq k$ the random variable $e(V_i,V_j)$ has a binomial distribution
	$\Bin(m,q_{ij})$, where
		$$q_{ij}=\frac{2\nu_i \nu_j}{n^2-\sum_{l=1}^k\nu_l^2}\geq\frac{2\nu_i\nu_j}{n^2}.$$
Since we are assuming that $\nu_i,\nu_j=(1+o_k(1))\frac nk$, we have $q_{ij}\geq(2+o_k(1))/k^2$.
Thus, 
	$\Erw\brk{e(V_i,V_j)}=mq_{ij}\geq
		(1+o_k(1))dn/{k^2}.$
Hence, the Chernoff bound yields
	$$\pr\brk{e(V_i,V_j)\leq \frac{0.99dn}{2k^2}}\leq\exp\brk{-\frac{dn}{8\cdot 10^4k^2}}\leq\exp(-12 n/k).$$
Finally, the first assertion follows by taking a union bound over $i,j$.
The second assertion follows analogously.
\qed\end{proof}

\medskip\noindent{\em Proof of \Lem~\ref{Lem_size_of_W}.}
By \Lem~\ref{Lemma_fixEdgeCounts} we may disregard the case that
$\min_{1\leq i<j\leq k}e(V_i,V_j)\leq \frac{0.99dn}{k^2}$.
Thus, fix integers $(m_{ij})_{1\leq i<j\leq k}$ such that
	\begin{equation}\label{eqLem_size_of_W1}
	m_{ij}\geq\frac{0.99dn}{k^2}\quad\mbox{ and }\quad\sum_{1\leq i<j\leq k}m_{ij}=m.
	\end{equation}
Let $\cM$ be the event that $e(V_i,V_j)=m_{ij}$ for all $1\leq i<j\leq k$.

We need to get a handle on the random variables $(e(v,V_j))_{v\in V_i}$ 
	(i.e., the number of neighbors of $v$ in $V_j$)
	in the random graph $G'(\sigma)$.
Given that $\cM$ occurs we know that $\sum_{v\in V_i}e(v,V_j)=e(V_i,V_j)=m_{ij}$.
Furthermore, because $G'(\sigma)$ consists of $m$ \emph{independent} random edges $\vec e_1,\ldots,\vec e_m$,
given the event $\cM$ the $m_{ij}$ edges between $V_i$ and $V_j$ are chosen uniformly and independently.
Therefore, we can think of the vertices in $V_i$ as ``bins'' and of the $m_{ij}$ edges as randomly tossed ``balls''.
In particular, the average number of balls that each bin $v\in V_i$ receives is $m_{ij}/\nu_i$.
Crucially, these balls-and-bins experiments are independent for all $i,j$.

To analyze them, we are going to use 
\Cor~\ref{Cor_BallsAndBins}.
Thus, 
 consider a family $(b_{vj})_{v\in V,j\in\brk k\setminus\cbc{\sigma(v)}}$ of mutually
independent Poisson variables with means $\Erw\brk{b_{vj}}=m_{\sigma(v)j}/\nu_{\sigma(v)}.$
Then for any family $(t_{vj})_{v\in V,j\in\brk k\setminus\cbc{\sigma(v)}}$ of integers we have
	\begin{equation}\label{eqWPoisson}
		\pr\brk{\forall v,j:e(v,V_j)=t_{vj}|\cM}
	\leq\exp(o(n))\cdot
		\pr\brk{\forall v,j:b_{vj}=t_{vj}}.
	\end{equation}
In words, the joint probability that the random variables $(e(v,V_j))_{v\in V,j\in\brk k\setminus\cbc{\sigma(v)}}$ take certain values given that $\cM$ occurs is
dominated by the corresponding event for the random variables $(b_{vj})$.

If $|W_i|>\frac{n\ln\ln k}{k\ln k}$, then there are at least $N=\frac{n\ln\ln k}{k\ln k}$ vertices $v\in V_i$ such that $\min_{j\in\brk k\setminus\cbc i}e(v,V_j)<3\ell$.
Thus, let $\cW_i$ be the number of vertices $v\in V_i$ such that $\min_{j\in\brk k\setminus\cbc i}b_{vj}<3\ell$.
Then~(\ref{eqWPoisson}) yields
	\begin{equation}\label{eqWPoisson3}
	\pr\brk{|W_i|\geq N|\cM}\leq\exp(o(n))\cdot\pr\brk{\cW_i\geq N}.
	\end{equation}
Furthermore, because the random variables $(b_{vj})_{v\in V_i,j\in\brk k\setminus\cbc i}$ are mutually independent, $\cW_i$ is a binomial random variable
with mean
	$\Erw\brk{\cW_i}\leq\nu_i q_i$, where $q_i=\sum_{j\in\brk k\setminus\cbc i}\pr\brk{\Po(m_{ij}/\nu_i)\leq 3\ell}$.
Since $\nu_i=(1+o_k(1))n/k$ and $m_{ij}\geq0.99d n/k^2$, we have $\mu_{ij}/\nu_i\geq0.98d/k\geq1.95\ln k$.
Recalling that $\ell=\exp(-7)\ln k$, we find 
$\pr\brk{\Po(m_{ij}/\nu_i)\leq 3\ell}\leq k^{-1.9}$ and thus $q_i\leq (k-1)k^{-1.9}$.
Hence,
	\begin{equation}\label{eqWPoisson6}
	\Erw\brk{\cW_i}\leq(1+o_k(1))k^{-1.9}n\leq k^{-1.8}n.
	\end{equation}
Therefore, the Chernoff bound gives
	\begin{equation}\label{eqWPoisson4}
	\pr\brk{\cW_i\geq N}\leq\exp\brk{-N\ln\bcfr{k^{1.8}N}{\eul n}}\leq\exp(-20n/k).
	\end{equation}
Combining~(\ref{eqWPoisson3}) and~(\ref{eqWPoisson4}), we obtain
	\begin{equation}\label{eqWPoisson5}
	\pr\brk{|W_i|\geq N|\cM}\leq\exp(o(n)-20n/k)\leq\exp(-19n/k).
	\end{equation}

Now, consider the event that there are at least $\kappa=\lceil\ln^4k\rceil$ classes $i_1,\ldots,i_\kappa$ such that $|W_i|\geq N'=\frac{n}{k\ln^4k}$.
We have
	\begin{equation}\label{eqWPoisson7}
	\pr\brk{\cW_{i_j}\geq N'}\leq\exp\brk{-N'\ln\bcfr{k^{1.8}N'}{\eul n}}\leq\exp\brk{-\frac12N'\ln k},
	\end{equation}
Furthermore, because the random variables $\cW_{i_1},\ldots,\cW_{i_\kappa}$ are independent,
we obtain from~(\ref{eqWPoisson}) and (\ref{eqWPoisson7})
	\begin{eqnarray}\label{eqWPoisson8}
	\pr\brk{\abs{\cbc{i\in\brk k:|W_i|\geq N'}}\geq\kappa|\cM}&\leq&\bink{k}{\kappa}\exp\brk{-\frac{\kappa}2N'\ln k}\leq\exp(-20n/k).
	\end{eqnarray}

With respect to the event $|W|\geq nk^{-0.7}$, observe that by~(\ref{eqWPoisson6}) the sum $\cW=\sum_{i=1}^k\cW_i$
is stochastically dominated by a binomial random variable with mean $nk^{-0.8}$.
Therefore, by~(\ref{eqWPoisson}) and the Chernoff bound
	\begin{equation}\label{eqWPoisson9}
	\pr\brk{\abs W\geq nk^{-0.7}|\cM}\leq\pr\brk{\cW\geq n k^{-0.7}}\leq\exp(-n k^{-0.7})\leq\exp(-20n/k).
	\end{equation}
Finally, since the estimates~(\ref{eqWPoisson5}), (\ref{eqWPoisson8}), (\ref{eqWPoisson9}) hold
for all $\cM$, the assertion follows from Bayes' formula.
\qed

\subsection{Proof of \Lem~\ref{Lem_size_of_U}}\label{Sec_Lem_size_of_U}

We begin by estimating the number of edges between the sets $W_i$ and the color class $V_j$.
As before, we let that $V_i=\sigma^{-1}(i)$ for $i\in\brk k$ and
 $\nu_i=|V_i|=(1+o_k(1))n/k$ for $i=1,\ldots,k$.

\begin{lemma}\label{Lemma_eWiVi}
In $G'(\sigma)$ we have
	$$\pr\brk{\max_{1\leq i<j\leq k}e(W_i,V_j)\geq \frac{2n\ln\ln k}{k}}\leq\exp(-11 n/k).$$
\end{lemma}
\begin{proof}
Fix $1\leq i<j\leq k$.
We begin by proving the following statement.
\begin{equation}\label{eqeWiVj1}
\parbox[c]{13.5cm}{For any set $S\subset V_i$ of size $|S|\leq\frac{n\ln\ln k}{k\ln k}$ we have $\pr\brk{e(S,V_j)> 2n\ln\ln k/k}\leq\exp(-13n/k)$.}
\end{equation}
Indeed, for any set $S$ as above the number $e(S,V_j)$ of edges $\vec e_i$ that join $S$ to $V_j$ has a binomial distribution
	$\Bin(m,q_{j,S})$, where
	$$q_{j,S}=\frac{2\nu_j|S|}{n^2-\sum_{l=1}^k\nu_l^2}\leq\frac{3\ln\ln k}{k^2\ln k};$$
the last inequality follows from our assumption that $\nu_l=(1+o_k(1))n/k$ for all $l\in\brk k$.
Hence,
	$$\Erw\brk{e(S,V_j)}=m q_{j,S}\leq\frac{3d\ln\ln k}{2k^2\ln k}\cdot n\leq\frac{3\ln\ln k}{2k}\cdot n$$
Thus, (\ref{eqeWiVj1}) follows from the Chernoff bound.
Taking the union bound over all possible sets $S$ of size $|S|\leq\frac{n\ln\ln k}{k\ln k}$, we obtain from~(\ref{eqeWiVj1})
\begin{equation}\label{eqeWiVj2}
\pr\brk{\exists S\subset V_i,\,|S|\leq\frac{n\ln\ln k}{k\ln k}:e(S,V_j)>\frac{2n\ln\ln k}k}\leq2^{\nu_i}\exp(-13n/k)\leq\exp(-12n/k).
\end{equation}
As $\pr\brk{|W_i|>\frac{n\ln\ln k}{k\ln k}}\leq\exp(-16n/k)$ by \Lem~\ref{Lem_size_of_W},
the assertion follows from~(\ref{eqeWiVj2}).
\qed\end{proof}

\begin{lemma}\label{Lemma_degreesFreak}
Let $T_{i}$ be the number of vertices $v\in V_i$ such that $\max_{j\neq i}e(v,V_j)>100\ln k$ and let $T=\sum_{i\in\brk k}T_i$.
Then in $G'(\sigma)$ we have
	$$\pr\brk{T>\frac{n}{4k\ln k}}\leq\exp(-10 n/k).$$
\end{lemma}
\begin{proof}
For an integer vector $\vec m=(m_{ij})_{1\leq i<j\leq k}$ let $\cE_{\vec m}$ be the event
that $e(V_i,V_j)=m_{ij}$ for all $1\leq i<j\leq k$.
Set $m_{ji}=m_{ij}$ for $1\leq i<j\leq k$.
By \Lem~\ref{Lemma_fixEdgeCounts} we may confine ourselves to the case that $e(V_i,V_j)\leq\frac{2dn}{k^2}$ for all $i\neq j$.
Thus, fix any $\vec m$ such that $m_{ij}\leq\frac{2dn}{k^2}$ for all $i<j$.
Given $\cE_{\vec m}$, for each of the $m_{ij}$ edges between color classes $V_i$, $V_j$ the actual vertex
in $V_i$ that the edge is incident with is uniformly distributed.
Thus, we can think of the vertices $v\in V_i$ as bins and of edge $m_{ij}$ edges as balls of color $j$, and our goal
is to figure out the probability that bin $v$ contains more than $100\ln k$ balls  colored $j$ for some $j\neq i$.
Because we are conditioning on $\cE_{\vec m}$, these balls-and-bins experiments are independent for all color pairs $i\neq j$.

To study these balls-and-bins experiments we  use \Cor~\ref{Cor_BallsAndBins}.
Thus, let $(b_{vj})_{v\in V,j\in\brk k\setminus\cbc{\sigma(v)}}$ be a family of mutually independent Poisson variables
such that $\Erw[b_{vj}]=m_{ij}/\nu_i$ for all $v\in V_i$, $j\in\brk k\setminus\cbc i$.
In addition, let $\cT_i$ be the number of vertices $v\in V_i$ such that $\max_{j\neq i}b_{vj}>100\ln k$ and let $\cT=\sum_{i=1}^k\cT_i$.
Then \Cor~\ref{Cor_BallsAndBins} gives
	\begin{equation}\label{eqdegreesFreak1}
	\pr\brk{T>\frac{n}{4k\ln k}\,\big|\,\cE_{\vec m}}
		\leq\exp(o(n))\cdot	\pr\brk{\cT>\frac{n}{4k\ln k}}
	\end{equation}

To complete the proof we need to bound $\pr\brk{\cT>\frac{n}{4k\ln k}}$.
For each vertex $v\in V_i$ and each $j\neq i$ we have
	$\Erw\brk{b_{vj}}= m_{ij}/\nu_i\leq\frac{2dn}{k^2\nu_i}\leq3\ln k$.
Hence, by Stirling's formula
	$$\pr\brk{b_{vj}>100\ln k}\leq\sum_{s>100\ln k}\Erw\brk{b_{vj}}^{s}/s!\leq k^{-90}.$$
Because the random variables $b_{vj}$ are mutually independent, $\cT$ is a sum of independent Bernoulli random variables.
Applying the union bound, we thus have
	\begin{equation}\label{eqdegreesFreak3}
	\pr\brk{\max_{j\neq\sigma(v)}b_{vj}>100\ln k}\leq k^{-89}\qquad\mbox{ for any $v\in V$}.
	\end{equation}
Therefore, (\ref{eqdegreesFreak3}) shows that $\cT$ is stochastically dominated by a binomial random variable $\Bin(n,k^{-89})$.
Consequently, the Chernoff bound yields
	\begin{equation}\label{eqdegreesFreak4}
	\pr\brk{\cT>\frac{n}{4k\ln k}}\leq\pr\brk{\Bin(n,k^{-89})>\frac{n}{4k\ln k}}\leq\exp(-20n/k).
	\end{equation}
Finally, combining~(\ref{eqdegreesFreak1}) and 
	(\ref{eqdegreesFreak4}) yields the assertion.		
\qed\end{proof}

\medskip\noindent{\em Proof of \Lem~\ref{Lem_size_of_U}.}
Let $\vec d=(d_{vj})_{v\in V,j\in\brk k\setminus\cbc{\sigma(v)}}$ be an integer vector.
Moreover, let $\cE_{\vec d}$ be the event that $e(v,V_j)=d_{vj}$ for all $v\in V$, $j\neq\sigma(v)$.
We are going to estimate the size of $U$ given that $\cE_{\vec d}$ occurs for a vector $\vec d$
that is ``compatible'' with the properties established in \Lem s~\ref{Lemma_fixEdgeCounts}--\ref{Lemma_degreesFreak}.
More precisely, we call $\vec d$ \emph{feasible} if the following conditions are satisfied.
\begin{enumerate}
\item[i.] For all $i\neq j$ we have $m_{ij}=\sum_{v\in V_i}d_{vj}\geq\frac{dn}{2k^2}$.
	Moreover, $m_{ij}=m_{ji}$.
\item[ii.] 
	 For all $i\neq j$ we have $w_{ij}=\sum_{v\in V_i:d_{vj}\leq3\ell}\leq\frac{2n\ln\ln k}{k}$.
\item[iii.]
	Let $\cT$ be the set of all vertices $v$ such that $\max_{j\neq\sigma(v)}d_{vj}>100\ln k$.
	Then $\abs\cT\leq\frac n{4k\ln k}$.
\end{enumerate}
By \Lem s~\ref{Lemma_fixEdgeCounts}--\ref{Lemma_degreesFreak}, we just need to show that
for any feasible $\vec d$ we have
	\begin{equation}\label{eqLem_size_of_U1}
	\pr\brk{\abs U>\frac{n\ln\ln k}{k\ln k}\,\big|\,\cE_{\vec d}}\leq\exp(-10n/k).
	\end{equation}

Given the event $\cE_{\vec d}$, the total number $m_{ij}$ of $V_i$-$V_j$-edges is fixed.
So is the number $w_{ji}$ of $W_j$-$V_i$ edges.
What remains random is how these edges are matched to the vertices in $V_i$.
More specifically, think of the $W_j$-$V_i$-edges as black balls, of the $V_j\setminus W_j$-$V_i$-edges are white balls,
and of the vertices $v\in V_i$ as bins.
Each bin $v$ has a capacity $d_{vj}$.
Now, the balls are tossed randomly into the bins, and our objective is to figure out the number of bins that receive more than $\ell$ black balls.
Observe that these numbers are independent for all pairs $i\neq j$ of colors.

To estimate this probability, consider a family $(b_{vj})_{v\in V,j\in\brk k\setminus\cbc{\sigma(v)}}$ of independent binomial
random variables such that $b_{vj}$ has distribution $\Bin(d_{vj},w_{ji}/m_{ji})$.
Let $\cB$ be the event that $\sum_{v\in V_i}b_{vj}=w_{ji}$ for all $i\neq j$.
Furthermore, let $\cU$ be the number of vertices $v$ such that $\max_{j\neq\sigma(v)}b_{vj}>\ell$.
Then
	\begin{equation}\label{eqLem_size_of_U2}
	\pr\brk{\abs U>\frac{n\ln\ln k}{k\ln k}\,\big|\,\cE_{\vec d}}=	\pr\brk{\cU>\frac{n\ln\ln k}{k\ln k}\,\big|\,\cB}\leq\frac{\pr\brk{\cU>\frac{n\ln\ln k}{k\ln k}}}{\pr\brk\cB}.
	\end{equation}

The sums $\sum_{v\in V_i}b_{vj}$ are binomial random variables $\Bin(m_{ij},w_{ji}/m_{ji})$.
Moreover, they are independent for all $i\neq j$.
Therefore, Stirling's formula yields
	\begin{equation}\label{eqLem_size_of_U3}
	\pr\brk\cB=\prod_{i\neq j}\pr\brk{\Bin(m_{ij},w_{ij}/m_{ij})=m_{ij}}=\Theta(n^{-k(k-1)/2})=\exp(o(n)).
	\end{equation}
	
Let $v\in V_i$ be a vertex such that for color $j\neq i$ we have $d_{vj}\leq100\ln k$.
Then our assumptions i.\ and ii.\ on $\vec d$ ensure that
	$\Erw\brk{b_{vj}}=\frac{w_{ji}d_{vj}}{m_{ji}}\leq300\ln\ln k.$
Therefore, by the Chernoff bound
	$$\pr\brk{b_{vj}\geq\ell}\leq\exp\brk{-\ell\cdot\ln\bcfr{\ell}{\eul\cdot \Erw\brk{b_{vj}}}}\leq k^{-100}.$$
Hence, taking the union bound, we find
	\begin{equation}\label{eqLem_size_of_U3}
	\pr\brk{\max_{j\neq\sigma(v)} b_{vj}\geq\ell}\leq k^{-99}\qquad\mbox{if }\max_{j\neq\sigma(v)}d_{vj}\leq100\ln k.
	\end{equation}
Let $\cU'$ be the number of vertices $v$ such that $\max_{j\neq\sigma(v)} b_{vj}\geq\ell$ while $\max_{j\neq\sigma(v)}d_{vj}\leq100\ln k$.
Because the random variables $b_{vj}$ are independent, (\ref{eqLem_size_of_U3}) implies that $\cU'$ is stochastically dominated by
a binomial random variable $\Bin(n,k^{-99})$.
Therefore, the Chernoff bound gives
	\begin{equation}\label{eqLem_size_of_U4}
	\pr\brk{\cU'\geq\frac{n\ln\ln k}{2k\ln k}}\leq\pr\brk{\Bin(n,k^{-99})\geq\frac{n\ln\ln k}{2k\ln k}}\leq\exp(-11 n/k).
	\end{equation}
As $\cU\leq\cT+\cU'\leq\cU'+\frac n{4k\ln k}$ by our assumption iii.\ on $\vec d$, (\ref{eqLem_size_of_U4})
implies that $\pr\brk{\cU\geq\frac{n\ln\ln k}{k\ln k}}\leq\exp(-11 n/k)$.
Thus, the assertion follows from~(\ref{eqLem_size_of_U2}) and~(\ref{eqLem_size_of_U3}).
\qed

\section{Proof of \Prop~\ref{Prop_noCover}}\label{Sec_Prop_noCover}

Throughout this section, we 
let $\zeta:V\ra\cbc{0,1,\ldots,k}$,
$V_i=\zeta^{-1}(i)$ and $\nu_i=|V_i|$ for $i=0,1,\ldots,k$.
In addition, we let $\alpha_i=\nu_i/n$.
We always assume that the conditions of \Prop~\ref{Prop_noCover} hold, namely
\begin{description}
\item[Z1.] $|\zeta^{-1}(0)|\leq nk^{-2/3}$.
\item[Z2.]  $|\zeta^{-1}(i)|=(1+o_k(1))n/k$ for all $i\in\brk k$.
\item[Z3.]  There are no more than $\ln^9k$ indices $i\in\brk k$ such that $|\zeta^{-1}(i)-n/k|>n/(k\ln^3k)$.
\end{description}
In addition, we assume that $d=2k\ln k-\ln k-c$ for some $0\leq c\leq 4$.

\subsection{Counting covers}

To prove \Prop~\ref{Prop_noCover} we perform a first moment argument over the number of covers $\zeta$.
Let $\cI_\zeta$ be the event that $V_1,\ldots,V_k$ are independent sets in $G'(n,m)$.
Moreover, let $\cC_\zeta$ be the event that $\zeta$ is a $k$-cover in $G'(n,m)$.
Clearly, $\cC_\zeta\subset\cI_\zeta$, and we begin
begin by estimating the probability the latter event.
Let $F_\zeta=\sum_{j=1}^k\alpha_j^2.$

\begin{lemma}\label{Lemma_respectsClasses}
We have
	$\frac1n\ln\pr\brk{\cI_\zeta}= \frac d2\ln(1-F_\zeta)$.
\end{lemma}
\begin{proof}
For each of the edges $\vec e_i$ the probability of joining two vertices in $V_j$ is $(\nu_j/n)^2=\alpha_j^2$.
Hence, the probability that $\vec e_i$ does not fall inside any of the classes $V_1,\ldots,V_k$ is equal to $1-F_\zeta$.
Thus, the assertion follows from the independence of  $\vec e_1,\ldots,\vec e_m$.
\qed\end{proof}

\noindent
In \Sec~\ref{Sec_isCover} we are going to establish the following estimate of the probability of $\cC_\sigma$.

\begin{lemma}\label{Lemma_isCover}
We have
	$\frac1n\ln\pr\brk{\cC_\zeta|\cI_\zeta}\leq
		\sum_{i=0}^k\alpha_i\ln p_i+o(1)$,
where
	\begin{eqnarray*}
	p_0&=&\sum_{i,j\in\brk k:i\neq j}\bc{\frac12+\frac{\alpha_j d}{1-F_\zeta}}\exp\brk{-\frac{(\alpha_i+\alpha_j)d}{1-F_\zeta}},\\
	p_i&=&\prod_{j\in\brk k\setminus\cbc i}1-\bc{1+\frac{\alpha_j d}{1-F_\zeta}}\exp\brk{-\frac{\alpha_j d}{1-F_\zeta}}\quad\mbox{ for }i=1,\ldots,k.
	\end{eqnarray*}
\end{lemma}

\medskip\noindent{\em Proof of \Prop~\ref{Prop_noCover}.}
Let $A$ be the set of all vectors
	$\vec\alpha=(\alpha_0,\ldots,\alpha_k)\in\brk{0,1}^{k+1}$ that satisfy the following three conditions (cf.\ {\bf Z1--Z3}):
\begin{enumerate}
\item[i.] $\sum_{i=0}^k\alpha_i=1$,
\item[ii.] We have $\alpha_0\leq k^{-2/3}$  and $\alpha_i=(1+o_k(1))/k$ for $i=1,\ldots,k$. Indeed,
		there are no more than $K=\ln^9k$ indices $i\in\brk k$ such that $|\alpha_i-1/k|>k^{-1}\ln^{-3}k$.
\item[iii.] $\alpha_in$ is an integer for $i=0,1,\ldots,k$.
\end{enumerate}
For $\vec\alpha\in A$ let $\cS_{\vec\alpha}$ be the set of all maps $\zeta:V\ra\cbc{0,1,\ldots,k}$
such that $|\zeta^{-1}(i)|=\alpha_i n$ for all $i$.
Then
	\begin{eqnarray}\nonumber
	\cS_{\vec\alpha}&=&\bink{n}{\alpha_0n,\ldots,\alpha_kn}=\bink{n}{\alpha_0 n}\cdot\bink{(1-\alpha_0)n}{\alpha_1 n,\ldots,\alpha_k n}
		\leq\bink{n}{\alpha_0 n}\cdot k^{(1-\alpha_0)n}.
	\end{eqnarray}
Hence, Stirling's formula yields
	\begin{eqnarray}\label{eqCoverCount1}
	\frac1n\ln\cS_{\vec\alpha}&\leq&-\alpha_0\ln\alpha_0-(1-\alpha_0)\ln(\bc{1-\alpha_0}/{k}).
	\end{eqnarray}

\Lem s~\ref{Lemma_respectsClasses} and~\ref{Lemma_isCover} show that for any $\zeta\in\cS_{\vec\alpha}$,
	\begin{eqnarray*}
	\frac1n\ln\pr\brk{\cC_\zeta}&\leq&o(1)+\frac d2\ln(1-F_\zeta)+\sum_{i=0}^k\alpha_i\ln p_i.
	\end{eqnarray*}
Given the value of $\alpha_0$, the sum $F_\zeta=\sum_{i=1}^k\alpha_i^2$ is minimized if $\alpha_i=(1-\alpha_0)/k$ for all $i\in\brk k$.
Thus,
	\begin{eqnarray}\label{eqCoverCount2}
	\frac1n\ln\pr\brk{\cC_\zeta}&\leq&o(1)+\frac d2\ln(1-(1-\alpha_0)^2/k)+\sum_{i=0}^k\alpha_i\ln p_i.
	\end{eqnarray}
Using the approximation $\ln(1-z)=-z-z^2/2+O(z^3)$ and recalling that $d=2k\ln k-\ln k-c$, we see that
	\begin{eqnarray}\nonumber
	\frac d2\ln(1-(1-\alpha_0)^2/k)&=&
		-(1-\alpha_0)^2\ln k\\
		&&\quad+\frac{(1-\alpha_0)^2\ln k}{2k}+\frac{c(1-\alpha_0)^2}{2k}-\frac{(1-\alpha_0)^4\ln k}{2k}+O_k(k^{-1.9})\nonumber\\
			&=&-(1-2\alpha_0)\ln k+\frac{c}{2k}+o_k(1/k)\qquad\mbox{[as $\alpha_0\leq k^{-2/3}$ by ii.]}.
			\label{eqCoverCount3}
	\end{eqnarray}

Furthermore, because $F_\zeta\in(0,1)$ and as $\ln(1-z)\leq -z$ for all $z\in(0,1)$, we get
	\begin{eqnarray}\nonumber
	\sum_{i=1}^k\alpha_i\ln p_i&\leq&-\sum_{i,j\in\brk k:i\neq j}\alpha_i(1+\alpha_j d)\exp(-\alpha_j d)\\
		&=&-\sum_{j\in\brk k}(1-\alpha_0-\alpha_j)(1+\alpha_j d)\exp(-\alpha_j d).\label{eqCoverCount10}
	\end{eqnarray}
Since $\alpha_j=(1+o_k(1))/k$ for all $j\in\brk k$ by ii.\ and as $d=2k\ln k-O_k(\ln k)$, (\ref{eqCoverCount10}) yields
	\begin{eqnarray}\label{eqCoverCount11}
	\sum_{i=1}^k\alpha_i\ln p_i&\leq&O_k(k^{-1.9})-(1-\alpha_0)\sum_{j\in\brk k}(1+\alpha_j d)\exp(-2\alpha_j k\ln k).
	\end{eqnarray}
Moreover, applying condition ii., we obtain from~(\ref{eqCoverCount11})
	\begin{eqnarray}\nonumber
	\sum_{i=1}^k\alpha_i\ln p_i&\leq&O_k(k^{-1.9})+O(K\ln k)\cdot \exp(-2(1+o_k(1))\ln k)\\[-4mm]
			&&\quad-(1-\alpha_0)(k-K)(1+2\ln k+O_k(1/\ln^2 k))\exp(-2(1+O_k(\ln^{-3}k))\ln k)\nonumber\\
		&\leq&o_k(1/k)-(1-\alpha_0)\cdot\frac{1+2\ln k}{k}\qquad\qquad\mbox{[as $K\leq k^{0.01}$]}\nonumber\\
		&\leq&o_k(1/k)-\frac{1+2\ln k}{k}\qquad\qquad\qquad\qquad\mbox{[as $\alpha_0\leq k^{-2/3}$ by ii.]}
		\label{eqCoverCount12}
	\end{eqnarray}
Further, again because $F_\zeta\in(0,1)$  we have
	\begin{eqnarray}\nonumber
	p_0&\leq&\frac12\sum_{i,j\in\brk k:i\neq j}\bc{1+2\alpha_j d}\exp\brk{-(\alpha_i+\alpha_j)d},\\
		&\leq&O_k(k^{-3+o_k(1)}K)+
			\frac{k(k-1)}2\brk{1+4\ln k+O_k(\ln^{-2}k)}
				\exp\brk{-4\ln k+O(\ln^{-2}k)}
				\quad\mbox{[by condition ii.]}\nonumber\\
		&\leq&\frac{1+4\ln k+O_k(\ln^{-1}k)}{2k^{2}}.
		\nonumber \label{eqCoverCount14a}
	\end{eqnarray}	
Hence,
	\begin{eqnarray}		\label{eqCoverCount14}
	\alpha_0\ln p_0&\leq&\alpha_0\ln\bcfr{1+4\ln k}{2k^{2}}+\alpha_0\cdot o_k(1).
	\end{eqnarray}
Plugging~(\ref{eqCoverCount3}), ~(\ref{eqCoverCount12}) and (\ref{eqCoverCount14}) into~(\ref{eqCoverCount2}), we obtain
	\begin{eqnarray}\nonumber
	\frac1n\ln\pr\brk{\cC_\zeta}&\leq&
		\frac{c}{2k}-(1-2\alpha_0)\ln k-\frac{1+2\ln k}{k}+\alpha_0\ln\bcfr{1+4\ln k}{2k^2}+\alpha_0\cdot o_k(1)+o_k(1/k)\\
		&=&\frac{c-2-4\ln k}{2k}-\ln k+\alpha_0\ln\bcfr{1+4\ln k}{2}+\alpha_0\cdot o_k(1)+o_k(1/k).
		\label{eqCoverCount15}
	\end{eqnarray}

Finally, combining~(\ref{eqCoverCount1}) and~(\ref{eqCoverCount15}), we get
	\begin{eqnarray}\nonumber
	\frac1n\ln(\abs{\cS_{\vec\alpha}}\cdot\pr\brk{\cC_\zeta})&\leq&
		\frac{c-2-4\ln k}{2k}-\alpha_0\ln\bcfr{2k\alpha_0}{1+4\ln k}\\
		&&\quad-(1-\alpha_0)\ln(1-\alpha_0)+\alpha_0\cdot o_k(1)+o_k(1/k)\nonumber\\
		&\leq&\frac{c-2-4\ln k}{2k}+\alpha_0\brk{1-\ln\bcfr{2k\alpha_0}{1+4\ln k}+o_k(1)}+o_k(1/k).
		\label{eqCoverCount16}
	\end{eqnarray}
Elementary calculus shows that the function $\alpha_0\in(0,1)\mapsto -\alpha_0(1-\ln\frac{2k\alpha_0}{1+4\ln k}+o_k(1))$ attains its maximum at
	$\alpha_0=(1+o_k(1))\frac{1+4\ln k}{2k}$.
Hence, (\ref{eqCoverCount16}) yields
	\begin{eqnarray}\label{eqFinal1}
	\frac1n\ln(\abs{\cS_{\vec\alpha}}\cdot\pr\brk{\cC_\zeta})&\leq&\frac{c-1+o_k(1)}{2k}. 
	\end{eqnarray}

To complete the proof, consider for any $\vec\alpha\in A$ the number $\Sigma_{\vec\alpha}$ of $k$-covers $\zeta$ of $\Gnm$ such that
$\abs{\zeta^{-1}(i)}=\alpha_i$ for all $i$.
Then~(\ref{eqFinal1}) implies that
$\frac1n\ln\Erw[\Sigma_{\vec\alpha}]\leq\frac{c-1}{2k}-o_k(1)$ for all $\vec\alpha\in A$.
Hence, there is $0<\eps_k=o_k(1)$ such that for $c<1-\eps_k$ we have
	\begin{eqnarray}\label{eqFinal2}
	\Erw[\Sigma_{\vec\alpha}]\leq\exp\brk{\frac{c-1}{2k}-o_k(1)}\leq\exp(-\eps_k/2)=\exp(-\Omega(n)).
	\end{eqnarray}
Since condition iii.\ ensures that $|A|\leq n^{k}=\exp(o(n))$, the assertion follows from (\ref{eqFinal2}) by taking the
union bound over all $\vec\alpha\in A$ and applying \Lem~\ref{Lemma_independentEdges}.
\qed

\subsection{Proof of \Lem~\ref{Lemma_isCover}}\label{Sec_isCover}

Given $\cI_\zeta$, the pairs $\vec e_1,\ldots,\vec e_m$ that constitute the random graph $G'(n,m)$ are simply
distributed uniformly and independently over the set of all $n^2(1-F_\zeta)$ possible pairs that do not join two vertices
in the same class $V_i$ for $i=1,\ldots,k$.
For each vertex $v$ and each $j\in\cbc{0,1,\ldots,k}$ let
	$d_{v,j}$ be the number of pairs $\vec e_i$ such that $\vec e_i$ contains $v$ together with a vertex from $V_j$.
Clearly, given $\cI_\zeta$ we have $d_{v,j}=0$ for all $v\in V_j$, $j\in\brk k$.

It seems reasonable to expect that the $d_{v,j}$ are asymptotically independent Poisson random variables.
To state this precisely, consider a family $(b_{vj})_{v\in V,j\in\cbc{0,1,\ldots,k}}$ of independent Poisson random variables with means
	$$\Erw\brk{b_{vj}}=\frac{\alpha_j d}{1-F_\zeta}.$$
Let $\cB_\zeta$ be the event that
\begin{enumerate}
\item[i.] for any $v\in V_0$ there exist $i,j\in\brk k$, $i\neq j$ such that $b_{vi}=0$ and $b_{vj}\leq1$ and
\item[ii.] for any $1\leq i<j\leq k$ and any $v\in V_i$ we have $b_{vj}>1$.
\end{enumerate}
The key step in the proof (somewhat reminiscent of the Poisson cloning model~\cite{Kim}) is to establish the following.

\begin{lemma}\label{Lemma_Poisson}
We have $\pr\brk{\cC_\zeta|\cI_\zeta}\leq\exp(o(n))\cdot\pr\brk{\cB_\zeta}$.
\end{lemma}

To prove \Lem~\ref{Lemma_Poisson} we consider a further event.
Set
	$B_{ij}=\sum_{v\in V_i}b_{vj}$ for $i,j\in\cbc{0,1,\ldots,k}$, $(i,j)\neq(0,0)$.
Being sums of independent Poisson variables, the random variables $B_{ij}$ are Poisson as well, with means
	$$\Erw[B_{ij}]=\Erw\brk{B_{ji}}=\alpha_i\alpha_j dn/(1-F_\zeta)\qquad(0\leq i<j\leq k).$$
In addition, let
	$B_{00}$ be a random variable that is independent of all of the above
	such that $\frac12 B_{00}$ has distribution $\Po(\alpha_0^2m/(1-F_\zeta))$.
(In particular, $B_{00}$ takes even values only.)
Now, let $\cV$ be the event that
\begin{enumerate}
\item[i.] $B_{ij}=B_{ji}$ for all $i\neq j$ and
\item[ii.] $\frac12B_{00}+\sum_{0\leq i<j\leq k}B_{ij}=m$.
\end{enumerate}

\begin{lemma}\label{Lemma_PoissonConditioning}
We have $\pr\brk{\cV}=\exp(o(n))$.
\end{lemma}
\begin{proof}
Since
	$$\Erw\brk{\frac{B_{00}}2+\sum_{1\leq i<j\leq k}B_{ij}}=\frac{dn}{2(1-F(\zeta))}\brk{\alpha_0^2+\sum_{i\neq j}\alpha_i\alpha_j}=m,$$
there exist integers $\beta_{ij}=\Erw[B_{ij}]+O(1)$ such that $\beta_{ij}=\beta_{ji}$
and $\frac12\beta_{00}+\sum_{0\leq i<j\leq k}\beta_{ij}=m$.
Clearly,
	\begin{equation}\label{eqPoissonConditioning}
	\pr\brk\cV\geq\pr\brk{B_{ij}=\beta_{ij}\mbox{ for all }i,j}=\prod_{i,j}\pr\brk{B_{ij}=\beta_{ij}}.
	\end{equation}
Since $\beta_{ij}=\Erw\brk{B_{ij}}+O(1)$ and $B_{ij}$ is a Poisson variable, Stirling's formula yields $\pr\brk{B_{ij}=\beta_{ij}}=\Omega(n^{-1/2})$.
Therefore, (\ref{eqPoissonConditioning}) implies $\pr\brk\cV\geq\Omega(n^{-(k+1)^2/2})=\exp(o(n))$, as claimed.
\qed\end{proof}

\medskip\noindent{\em Proof of \Lem~\ref{Lemma_Poisson}.}
Let $\vec m=(m_{ij})_{i,j\in\cbc{0,1,\ldots,k}}$ be a family of non-negative integers such that
\begin{enumerate}
\item[a.] $m_{ij}=m_{ji}$ for all $i,j$,
\item[b.] $m_{ii}=0$ for $i\in\brk k$ and
\item[c.] $m_{00}+\sum_{0\leq i<j\leq k}m_{ij}=m$.
\end{enumerate}
Let $\cM_{\vec m}$ be the event that
	$$\sum_{v\in V_0}d_{v0}=2m_{00}\quad\mbox{ and }\quad\sum_{v\in V_i}d_{vj}=m_{ij}\mbox{ for any $0\leq i<j\leq k$.}$$
Analogously, let $\cM_{\vec m}'$ be the event that
	$$B_{00}=2m_{00}\quad\mbox{ and }\quad B_{ij}=m_{ij}\mbox{ for any $0\leq i<j\leq k$.}$$
	
We claim that for any $\vec m$ that satisfies a.--c.\ above we have
	\begin{equation}\label{eqLemma_Poisson1}
	\pr\brk{\cC_\zeta|\cM_{\vec m}}=\pr\brk{\cB_\zeta|\cM_{\vec m}'}.
	\end{equation}
Indeed, let either $i=j=0$ or $0\leq i<j\leq k$.
Given that $\cM_{\vec m}$ occurs, we can think of the $m_{ij}$ edges that join $V_i$ and $V_j$ as balls and of
the vertices $v\in V_i$ as bins.
Each ball is tossed into one of the bins randomly and independently, and these events are
independent for all $i,j$.
Thus, (\ref{eqLemma_Poisson1}) simply follows from the Poissonization of the balls and bins experiment (\Lem~\ref{Lemma_BallsAndBins}).

To complete the proof, we need to compare $\pr\brk{\cM_{\vec m}|\cI_\zeta}$ and $\pr\brk{\cM_{\vec m}'|\cV}$.
Because under the distribution $\pr\brk{\,\cdot\,|\cI_\zeta}$ the pairs $\vec e_1,\ldots,\vec e_m$ are simply chosen randomly
subject to the constraint that none of them joins two vertices in the same class $V_i$, $i\in \brk k$, we see that
	\begin{equation}\label{eqLemma_Poisson2}	
	\pr\brk{\cM_{\vec m}|\cI_\zeta}=\frac{m!}{m_{00}!\prod_{0\leq i<j\leq k}m_{ij}!}\cdot\bcfr{\alpha_0^2}{1-F_\zeta}^{m_{00}}
			\prod_{0\leq i<j\leq k}\bcfr{2\alpha_i\alpha_j}{1-F_\zeta}^{m_{ij}}.
	\end{equation}
(The factor of $2$ arises because $\vec e_1,\ldots,\vec e_m$ are {\em ordered} pairs.)
Furthermore, because $\cV$ provides that $B_{ij}=B_{ji}$ for all $i,j$, we have
	\begin{eqnarray}\nonumber
	\pr\brk{\cM_{\vec m}'|\cV}&=&\frac{\pr\brk{B_{00}=2m_{00}}\cdot\prod_{0\leq i<j\leq k}\pr\brk{B_{ij}=m_{ij}}}{\pr\brk{\cV}}.
	\end{eqnarray}
Thus, by \Lem~\ref{Lemma_PoissonConditioning}
	\begin{eqnarray}
	\pr\brk{\cM_{\vec m}'|\cV}&=&\exp(o(n))\cdot \pr\brk{B_{00}=2m_{00}}\cdot\prod_{0\leq i<j\leq k}\pr\brk{B_{ij}=m_{ij}}.
	\label{eqLemma_Poisson3}
	\end{eqnarray}
Since for $0\leq i<j\leq k$ the random variables $B_{ij}$ are Poisson with mean $\alpha_i\alpha_j dn/(1-F_\zeta)$, we have
	\begin{equation}\label{eqLemma_Poisson4}
	\pr\brk{B_{ij}=m_{ij}}=\frac{(\alpha_i\alpha_j dn/(1-F_\zeta))^{m_{ij}}}{m_{ij}!\exp(\alpha_i\alpha_j dn/(1-F_\zeta))}
		=\bcfr{2\alpha_i\alpha_j}{1-F_\zeta}^{m_{ij}}
			\frac{m^{m_{ij}}}{m_{ij}!\exp(2\alpha_i\alpha_j m/(1-F_\zeta))}.
	\end{equation}
Similarly,
	\begin{equation}\label{eqLemma_Poisson5}
	\pr\brk{B_{00}=2m_{00}}=\frac{(\alpha_0^2m/(1-F_\zeta))^{m_{00}}}{m_{00}!\exp(\alpha_0^2m/(1-F_\zeta))}
		=\bcfr{\alpha_0^2}{1-F_\zeta}^{m_{00}}
			\frac{m^{m_{00}}}{m_{00}!\exp(\alpha_0^2 m/(1-F_\zeta))}.
	\end{equation}
Combining~(\ref{eqLemma_Poisson2})--(\ref{eqLemma_Poisson5}), we obtain from Stirling's formula
	\begin{eqnarray}\nonumber
	\frac{\pr\brk{\cM_{\vec m}|\cI_\zeta}}{\pr\brk{\cM_{\vec m}'|\cV}}&=&
		\frac{m!\exp(m(1-F_\zeta)^{-1}(\alpha_0^2+2\sum_{0\leq i<j\leq k}\alpha_i\alpha_j))}{\exp(o(n))m^m}\\
		&=&\frac{m!\exp(m+o(n))}{m^m}=\exp(o(n)).\label{eqLemma_Poisson6}
	\end{eqnarray}

Finally, combining~(\ref{eqLemma_Poisson1}) and~(\ref{eqLemma_Poisson6}) we conclude that for any ${\vec m}$ that satisfies a.--c.\ we have
	\begin{eqnarray*}
	\pr\brk{\cC_\zeta\cap\cM_{\vec m}|\cI_\zeta}&=&\pr\brk{\cC_\zeta|\cM_{\vec m}}\cdot\pr\brk{\cM_{\vec m}|\cI_\zeta}
				\qquad\qquad\mbox{ [as $\cM_{\vec m}\subset\cI_\zeta$]}\\
		&=&\exp(o(n))\pr\brk{\cB_\zeta|\cM_{\vec m}'}\cdot\pr\brk{\cM_{\vec m}'|\cV}\\
		&=&\exp(o(n))\pr\brk{\cB_\zeta\cap \cM_{\vec m}'}/\pr\brk{\cV}\\
		&=&\exp(o(n))\pr\brk{\cB_\zeta\cap\cM_{\vec m}'}\qquad\qquad\quad\mbox{[due to \Lem~\ref{Lemma_PoissonConditioning}].}
	\end{eqnarray*}
Summing over all possible ${\vec m}$ completes the proof.
\qed

\medskip\noindent{\em Proof of \Lem~\ref{Lemma_isCover}.}
We are going to bound the probability of the event $\cB_\zeta$.
For $v\in V_0$ we have
	\begin{eqnarray*}
	\pr\brk{\exists 1\leq i<j\leq k:b_{vi}=b_{vj}=0}&\leq&\sum_{1\leq i<j\leq k}\pr\brk{b_{vi}=b_{vj}=0}
			+\sum_{i,j\in\brk k:i\neq j}\pr\brk{b_{vi}=0,\,b_{vj}=1}=p_0,
	\end{eqnarray*}
because the $b_{vi},b_{vj}$ are independent Poisson variables.
Similarly, if $v\in V_i$ for some $i\in\brk k$, then
	\begin{eqnarray*}
	\pr\brk{\forall j\in\brk k\setminus\cbc i:b_{vj}>1}&=&\prod_{j\in\brk k\setminus\cbc i}1-\pr\brk{b_{vj}\leq1}=p_i.
	\end{eqnarray*}
Due to the mutual independence of the $b_{vj}$, we thus obtain
	$\pr\brk{\cB_\zeta}=
			p_0^{\alpha_0 n}\prod_{i=1}^k p_i^{\alpha_i n}.\label{eqLemma_isCover1}$
Finally, the assertion follows from \Lem~\ref{Lemma_Poisson}.
\qed

\end{document}